\documentclass[12pt]{article}
\usepackage[latin1]{inputenc}
\usepackage{amssymb}
\usepackage{amsmath}
\usepackage{amsthm}
\usepackage{tikz}
\usepackage[colorlinks=true,linkcolor=blue,citecolor=magenta]{hyperref}

\definecolor{c0000ff}{RGB}{0,0,255}
\definecolor{c00ff00}{RGB}{0,255,0}
\definecolor{cff0000}{RGB}{255,0,0}

\theoremstyle{definition}

\newtheorem{thm}{Theorem}[section]

\newtheorem{prop}[thm]{Proposition}

\newtheorem{lem}[thm]{Lemma}

\newtheorem{remark}[thm]{Remark}

\newtheorem{defn}[thm]{Definition}

\newtheorem{cor}[thm]{Corollary}

\newtheorem{question}[thm]{Question}

\newcommand{\vsp}{\vspace{0.2cm}}

\date{}
\author{}

\begin{document}

\title{Uniformly perfect finitely generated simple left orderable groups}

\author{James Hyde, Yash Lodha, Andr\'{e}s Navas and Crist\'{o}bal Rivas}
\maketitle

\begin{abstract}
We show that the finitely generated simple left orderable groups $G_{\rho}$ constructed by the first two authors in \cite{HydeLodha} are uniformly perfect -- 
each element in the group can be expressed as a product 
of three commutators of elements in the group.
This implies that the group does not admit any homogeneous quasimorphism. Moreover, any nontrivial action of the group on the circle, that lifts to an action on the real line, admits a global fixed point.
Most strikingly, it follows that the groups are examples of \emph{left orderable monsters}, which means that any faithful action on the real line without a global 
fixed point is globally contracting. 
This answers Question $4$ from the 2018 ICM proceedings article of the third author.
(This question has also been answered simultaneously and independently, using completely different methods, by Matte Bon and Triestino in \cite{MatteBonTriestino}.)
To prove our results, we provide a certain \emph{characterisation of elements} of the group $G_{\rho}$ which is a useful new tool in the study of these examples.
\end{abstract}

\vsp\vsp


In $1980$ Rhemtulla asked whether there exist finitely generated simple left orderable groups (see \cite{HydeLodha} for a discussion around the history of the problem and references).
This question was answered in the affirmative by the first two authors in \cite{HydeLodha}.
The construction takes as an input a certain \emph{quasi-periodic} labelling $\rho$ of the set $\frac{1}{2}\mathbf{Z}$ which is a map $$\rho:\frac{1}{2}\mathbf{Z}\to \{a,b,a^{-1},b^{-1}\}$$ that satisfies a certain set of axioms.
(See the preliminaries section for details.)
Such labellings exist and are easy to construct explicitly.
For each such labelling $\rho$, one constructs an explicit group action $G_{\rho}<\textup{Homeo}^+(\mathbf{R})$ which is a finitely generated simple left orderable group. 

Given a group $G$ and an element $f\in [G,G]$, the integer 
$cl(f)$ is defined as the smallest $k$ such that $f$ can be expressed as a product of $k$ commutators of elements in $G$.
Our main theorem is the following:

\begin{thm}\label{main}
Let $\rho$ be a quasi-periodic labelling.
Then for each element $f\in G_{\rho}$, it holds that $cl(f)\leq 3$.
\end{thm}

Recall that a homogeneous quasimorphism is a quasimorphism $\phi:G\to \mathbf{R}$ with the property that the restriction of $\phi$ to any cyclic subgroup is a homomorphism.
As a consequence of Theorem \ref{main} we obtain that the stable commutator length vanishes, and hence the group does not admit any 
nontrivial homogeneous quasimorphism. Using the work of Ghys \cite{ghys}, this allows us to show the following.

\begin{cor}\label{maincorollary}
Let $\rho$ be a quasi-periodic labelling. Then every faithful action of $G_{\rho}$ on $\mathbf{S}^1$, that lifts to an action on the real line, admits a global fixed point on $\mathbf{S}^1$.
\end{cor}

Recall that for every action of a finitely generated group $G$ by orientation preserving homeomorphisms of the real line without global fixed points, there are one of three possibilities:
\begin{enumerate}
\item[(i)] There is a $\sigma$-finite measure $\mu$ that is invariant under the action.
\item[(ii)] The action is semiconjugate to a minimal action for which every small enough interval is sent into a sequence of intervals that converge to a point under well chosen group elements, but this
property does not hold for every bounded interval. (Here, by a semiconjugacy we roughly mean a factor action for which the factor map is a continuous, non decreasing, proper map of the real line.)
\item[(iii)] The action is globally contracting; more precisely, it is 
semiconjugate to a minimal one for which the contraction property above holds for all bounded intervals.
\end{enumerate}

We obtain the following as an immediate consequence of Corollary \ref{maincorollary}.

\begin{cor}\label{maincor1}
Let $\rho$ be a quasi-periodic labelling.
Then any faithful action of the group $G_{\rho}$ on $\mathbf{R}$ without global fixed points is of type $\textup{(iii)}$.
\end{cor}

This answers the following question of the third author.

\begin{question}
(Navas, ICM proceedings $2018$ Question $4$) Does there exist an infinite, finitely-generated group that acts on the real line all of whose actions by orientation-preserving homeomorphisms of the line without global fixed points are of type $\textup{(iii)}$?
\end{question}

\begin{remark}
The above question has been answered simultaneously and independently by Matte Bon and Triestino in \cite{MatteBonTriestino}.
They provide a new family of finitely generated simple left orderable groups, which are overgroups of the groups $G_{\rho}$, and 
prove the analog of Corollary \ref{maincorollary} for that family. Their methods are completely different from ours.
\end{remark}

Corollaries \ref{maincorollary} and \ref{maincor1} should be compared with similar theorems for lattices in higher rank simple 
Lie groups. For them, it is known that every action on the circle has a finite orbit; therefore, up to a finite index group, they admit a global 
fixed point \cite{BM,ghys2}. However, it is still unknown whether they admit nontrivial actions on the line or not, yet several definitive 
results are known \cite{witte1,witte2,witte3}. In case one of these lattices admits such an action, it is not hard to see that it would also 
provide an affirmative answer to the Question above (see \cite{GOD}).

The proof of Theorem \ref{main} uses the following new description of the group which is the main technical result of the article.
Let $\rho$ be a quasiperiodic labelling.
(Recall this notion from \cite{HydeLodha}, or see Definition \ref{labelling})
Given an $x\in \mathbf{R}$ and $n\in \mathbf{N}$, we define a word $\mathcal{W}(x,n)$ as follows. 
Let $y\in \frac{1}{2}\mathbf{Z}\setminus \mathbf{Z}$ such that $x\in [y-\frac{1}{2},y+\frac{1}{2})$.
Then we define $$\mathcal{W}(x,n)=\rho(y-\frac{1}{2}n)\rho(y-\frac{1}{2}(n-1))...\rho(y)...\rho(y+\frac{1}{2}(n-1))\rho(y+\frac{1}{2} n)$$

For each integer $n\in \mathbf{Z}$, we denote by $\iota_n$ as the 
unique orientation reversing isometry
$\iota_n:[n,n+1)\to (n,n+1]$.
For $x\in \mathbf{R}$, we define the map $\iota:\mathbf{R}\to \mathbf{R}$ as 
$$x\cdot \iota=x\cdot \iota_n\qquad \text{ such that }x\in [n,n+1)\text{ for }n\in \mathbf{Z}$$
In what follows by a \emph{countably singular} piecewise linear homeomorphism,
we mean a piecewise linear homeomorphism with a countable set of singularities (or breakpoints).

\begin{defn}\label{Krho}
Let $K_{\rho}$ be the set
of homeomorphisms $f\in \textup{Homeo}^+(\mathbf{R})$ satisfying the following:
\begin{enumerate}
\item $f$ is a countably singular piecewise linear homeomorphism of $\mathbf{R}$ with a discrete set of singularities that lie in $\mathbf{Z}[\frac{1}{2}]$.
\item $f'(x)$, wherever it exists, is an integer power of $2$.
\item There is a $k_f\in \mathbf{N}$ such that the following holds.
\begin{enumerate}
\item[3.a] Whenever $x,y\in \mathbf{R}$ satisfy that $$x-y\in \mathbf{Z}\qquad \mathcal{W}(x,k_f)=\mathcal{W}(y,k_f)$$ it holds that $$x-x\cdot f = y-y\cdot f$$
\item[3.b] Whenever $x,y\in \mathbf{R}$ satisfy that $$x-y\in \mathbf{Z}\qquad \mathcal{W}(x,k_f)=\mathcal{W}^{-1}(y,k_f)$$ it holds that $$x-x\cdot f=y'\cdot f-y' \qquad \text{ where }y'=y\cdot \iota$$
\end{enumerate}
\end{enumerate}
\end{defn}

It is not hard to check that $K_{\rho}$ is a group.

\begin{remark}\label{WidthRemark}
Note that given an element $f\in K_{\rho}$ and a number $k_f\in \mathbf{N}$ satisfying the conditions of Definition \ref{Krho},
any number $k_f'\in \mathbf{N}$ such that $k_f'>k_f$ also satisfies the conditions of the Definition.
\end{remark}

\begin{thm}\label{characterisation}
$K_{\rho}\cong G_{\rho}$.
\end{thm}

This characterisation provides a useful new definition of the groups $G_{\rho}$ as groups of homeomorphisms of the real line satisfying a natural set of criterion.
This also provides useful new structural results such as the following.

\begin{prop}\label{mainprop2}
Let $\rho$ be a quasi-periodic labelling.
Given any element $f\in G_{\rho}$, there are elements $g_1,g_2\in G_{\rho}$ such that the following holds:
\begin{enumerate}
\item $f=g_1g_2$.
\item $g_2$ is a commutator in $G_{\rho}$.
\item There is a subgroup $K<G_{\rho}$ such that $K\cong F'\oplus...\oplus F'$ and $g_1\in K$.
\end{enumerate}
\end{prop}

\section{Preliminaries}

All actions will be right actions, unless otherwise specified.
Given a group action $G<\textup{Homeo}^+(\mathbf{R})$ and a $g\in G$,
we denote by $Supp(g)$ the set $$Supp(g)=\{x\in \mathbf{R}\mid x\cdot g\neq x\}$$
Note that $Supp(g)$ is an open set, and that $\mathbf{R}$ can be replaced by another $1$-manifold.
A homeomorphism $f:[0,1]\to [0,1]$ is said to be \emph{compactly supported in} $(0,1)$
if $\overline{Supp(f)}\subset (0,1)$.
Similarly, a homeomorphism $f:\mathbf{R}\to \mathbf{R}$ is said to be \emph{compactly supported in} $\mathbf{R}$
if $\overline{Supp(f)}$ is a compact interval in $\mathbf{R}$.
A point $x\in \mathbf{R}$ is said to be a \emph{transition point} of $f$ if $$x\in \partial Supp(f)=\overline{Supp(f)}\setminus Supp(f)$$

Our construction uses in an essential way the structure and properties of Thompson's group $F$.
We shall only describe the features of $F$ here that we need, and we direct the reader to 
\cite{CannonFloydParry} and \cite{Belk} for more comprehensive surveys.
Recall that the group $\textup{PL}^+([0,1])$ is the group of orientation preserving piecewise linear homeomorphisms of $[0,1]$.
Recall that $F$ is defined as the subgroup of $\textup{PL}^+([0,1])$ that satisfy the following:
\begin{enumerate}
\item Each element has at most finitely many breakpoints. All breakpoints lie in the set of dyadic rationals, i.e. $\mathbf{Z}[\frac{1}{2}]$.
\item For each element, the derivatives, wherever they exist, are powers of $2$.
\end{enumerate}
By \emph{breakpoint} (or a singularity point) we mean a point where the derivative does not exist.
For $r,s\in \mathbf{Z}[\frac{1}{2}]\cap[0,1]$ such that $r<s$, we denote by $F_{[r,s]}$ the subgroup of elements whose support lies in $[r,s]$.
The following are well known facts that we shall need.
The group $F$ satisfies the following:
\begin{enumerate}
\item $F$ is $2$-generated.
\item For each pair $r,s\in \mathbf{Z}[\frac{1}{2}]\cap[0,1]$ such that $r<s$, the group $F_{[r,s]}$ is isomorphic to $F$ and hence is also $2$-generated.
\item $F'$ is simple and consists of precisely the set of elements $g\in F$ such that $\overline{Supp(g)}\subset (0,1)$.
\end{enumerate}

An interval $I\subseteq [0,1]$ is said to be a \emph{standard dyadic interval}, if it is of the form $[\frac{a}{2^n},\frac{a+1}{2^{n}}]$
such that $a,n\in \mathbf{N}, a<2^n-1$.
The following are elementary facts about the action of $F$ on the standard dyadic intervals.

\begin{lem}\label{TransitiveStandard}
Let $I,J$ be standard dyadic intervals in $(0,1)$.
Then there is an element $f\in F'$ such that:
\begin{enumerate} 
\item $I\cdot f=J$.
\item $f\restriction I$ is linear.
\end{enumerate}
\end{lem}
\begin{lem}\label{TransitiveStandard2}
Let $I_1,I_2$ and $J_1,J_2$ be standard dyadic intervals in $(0,1)$ such that $$sup(I_1)<inf(I_2)\qquad sup(J_1)<inf(J_2)$$
Then there is an element $f\in F'$ such that:
\begin{enumerate}
\item $I_1\cdot f=J_1$ and $I_2\cdot f=J_2$.
\item $f\restriction I_1$ and $f\restriction I_2$ are linear.
\end{enumerate}
\end{lem}

We fix $\iota:(0,1)\to (0,1)$ as the unique orientation reversing isometry.
We say that an element $f\in F$ is \emph{symmetric}, if $f=\iota\circ f \circ \iota$.
We say that a set $I\subset (0,1)$ is \emph{symmetric} if $I\cdot \iota=I$.
Note that given any symmetric set $I$ with nonempty interior, we can find a 
nontrivial symmetric element $f\in F'$ such that $Supp(f)\subset int(I)$.
We extend the map $\iota$ to $\mathbf{R}$ as follows. For each integer $n\in \mathbf{Z}$, we denote the 
unique orientation reversing isometry 
$$\iota_n:[n,n+1)\to (n,n+1]$$
For $x\in \mathbf{R}$, we define the map $\iota:\mathbf{R}\to \mathbf{R}$ as 
$$x\cdot \iota=x\cdot \iota_n\qquad \text{ such that }x\in [n,n+1)\text{ for }n\in \mathbf{Z}$$

In the paper we shall also use the notations $\iota_{[x,y)}:[x,y)\to (x,y]$ or $\iota_I:I\to I$
to denote the unique orientation reversing isometries between intervals of the form $[x,y)$ and $(x,y]$ (for $x,y\in \mathbf{R}$), or a compact subinterval $I$ of $\mathbf{R}$.
The usage of this notation will be made clear when it occurs. (Note that it differs from the $\iota$ defined above.)

\begin{defn}\label{H}
We fix an element $c_0\in F$ with the following properties:
\begin{enumerate}
\item The support of $c_0$ equals $(0,\frac{1}{4})$ and $x\cdot c_0>x$ for each $x\in (0,\frac{1}{4})$.
\item $c_0\restriction {(0,\frac{1}{16})}$ equals the map $t\to 2t$.
\end{enumerate}
Let $$c_1=\iota \circ c_0\circ \iota \qquad \nu_1=c_0c_1$$
Note that $\nu_1\in F$ is a symmetric element.
We define a subgroup $H$ of $F$ as $$H=\langle F',\nu_1\rangle$$
Finally, we fix $$\nu_2,\nu_3:[0,1]\to [0,1]$$ as chosen homeomorphisms whose supports are contained in $(\frac{1}{16},\frac{15}{16})$ and that generate the group  $F_{[\frac{1}{16},\frac{15}{16}]}$.
\end{defn}

The following is Lemma $2.4$ in \cite{HydeLodha}.

\begin{lem}\label{3gen}
$H$ is generated by $\nu_1,\nu_2,\nu_3$. $H'$ is simple and consists of precisely the set of elements of $H$ (or $F$) that are compactly supported in $(0,1)$.
In particular, $H'=F'$.
\end{lem} 

\begin{defn}\label{labelling}
We consider the additive group $\frac{1}{2}\mathbf{Z}=\{\frac{1}{2}k\mid k\in \mathbf{Z}\}$.
A \emph{labelling} is a map $$\rho:\frac{1}{2}\mathbf{Z}\to \{a,b,a^{-1},b^{-1}\}$$
which satisfies:
\begin{enumerate}
\item $\rho(k)\in\{a,a^{-1}\}$ for each $k\in \mathbf{Z}$.
\item $\rho(k)\in \{b,b^{-1}\}$ for each $k\in \frac{1}{2}\mathbf{Z}\setminus \mathbf{Z}$.
\end{enumerate}

We regard $\rho(\frac{1}{2}\mathbf{Z})$ as a bi-infinite word with respect to the usual ordering of the integers.
A subset $X\subseteq \frac{1}{2}\mathbf{Z}$ is said to be a \emph{block} if it is of the form $$\{k,k+\frac{1}{2},...,k+\frac{1}{2}n\}$$
for some $k\in \frac{1}{2}\mathbf{Z}, n\in \mathbf{N}$.
Note that each block is endowed with the usual ordering inherited from $\mathbf{R}$.
The set of blocks of $\frac{1}{2}\mathbf{Z}$ is denoted as $\mathbf{B}$.
To each block $X=\{k,k+\frac{1}{2},...,k+\frac{1}{2}n\}$, we assign a formal word $$W_{\rho}(X)=\rho(k)\rho(k+\frac{1}{2})...\rho(k+\frac{1}{2}n)$$
which is a word in the letters $\{a,b,a^{-1},b^{-1}\}$.
Such a formal word is called a \emph{subword} of the labelling.

Recall that given a word $w_1...w_n$ in the letters $\{a,b,a^{-1},b^{-1}\}$, the formal inverse of the word is $w_n^{-1}...w_1^{-1}$.
The formal inverse of $W_{\rho}(X)$ is denoted as $W_{\rho}^{-1}(X)$.

A labelling $\rho$ is said to be \emph{quasi-periodic} if the following holds:
\begin{enumerate}
\item For each block $X\in \mathbf{B}$, there is an $n\in \mathbf{N}$ such that whenever $Y\in \mathbf{B}$ is a block of size at least $n$,
then $W_{\rho}(X)$ is a subword of $W_{\rho}(Y)$.
\item For each block $X\in \mathbf{B}$, there is a block $Y\in \mathbf{B}$ such that $W_{\rho}(Y)=W_{\rho}^{-1}(X)$.
\end{enumerate}
Note that by \emph{subword} in the above we mean a string of consecutive letters in the word.
\end{defn}

A nonempty finite word $w_1...w_n$ for $w_i\in \{a,b,a^{-1},b^{-1}\}$ is said to be a \emph{permissible word} if $n$ is odd and the following holds.
For odd $i\leq n$ one has $w_i\in \{a,a^{-1}\}$, and for even $i\leq n$ one has $w_i\in \{b,b^{-1}\}$.

The following is Lemma $3.1$ in \cite{HydeLodha}.

\begin{lem}\label{quasi-periodicLabellings}
Given any permissible word $w_1...w_m$, there is a quasi-periodic labelling $\rho$ of $\frac{1}{2}\mathbf{Z}$ and a block $X\in \mathbf{B}$ satisfying that $W_{\rho}(X)=w_1...w_m$.
\end{lem}

Following \cite{HydeLodha}, we recall that to each labelling $\rho$, we 
associate a group $G_{\rho}<\textup{Homeo}^+(\mathbf{R})$ as follows.

\begin{defn}\label{SimpleGroup}
Let $H<\textup{Homeo}^+([0,1])$ be the group defined in Definition \ref{H}.
Recall from Lemma \ref{3gen} that the group $H$ is generated by the three elements $\nu_1,\nu_2,\nu_3$ defined in Definition \ref{H}.
In what appears below, by $\cong_T$ we mean that the restrictions are topologically conjugate via the unique orientation preserving isometry that maps $[0,1]$ to the respective interval.
We define the homeomorphisms $$\zeta_1,\zeta_2,\zeta_3,\chi_1,\chi_2,\chi_3:\mathbf{R}\to \mathbf{R}$$ as follows for each $i\in \{1,2,3\}$ and $n\in \mathbf{Z}$:
$$\zeta_i\restriction [n,n+1]\cong_{T}\nu_i \qquad \text{ if } \rho(n+\frac{1}{2})=b$$
$$\zeta_i\restriction [n,n+1]\cong_T(\iota\circ \nu_i  \circ \iota) \qquad \text{ if }\rho(n+\frac{1}{2})=b^{-1}$$
$$\chi_i\restriction [n-\frac{1}{2},n+\frac{1}{2}]\cong_T \nu_i \qquad \text{ if }\rho(n)=a$$
$$\chi_i\restriction [n-\frac{1}{2},n+\frac{1}{2}]\cong_T (\iota\circ \nu_i \circ \iota)\qquad \text{ if }\rho(n)=a^{-1}$$

The group $G_{\rho}$ is defined as $$G_{\rho}:=\langle \zeta_1,\zeta_2,\zeta_3,\chi_1,\chi_2,\chi_3\rangle<\textup{Homeo}^+(\mathbf{R})$$
We denote the above generating set of $G_{\rho}$ as $$\mathbf{S}_{\rho}:=\{\zeta_1,\zeta_2,\zeta_3,\chi_1,\chi_2,\chi_3\}$$
We also define subgroups $$\mathcal{K}:=\langle \zeta_1,\zeta_2,\zeta_3\rangle\qquad \mathcal{L}:=\langle \chi_1,\chi_2,\chi_3\rangle$$ of $G_{\rho}$ that are both isomorphic to $H$,
and $$\mathcal{K}'\cong \mathcal{L}'\cong F'$$
Note that the definition of $\mathcal{K},\mathcal{L}$ requires us to fix a labelling $\rho$ but we denote them as such for simplicity of notation.
\end{defn}

Note that the group $G_{\rho}$ is defined for every labelling $\rho$.
The following is proved in \cite{HydeLodha}.

\begin{thm}
Let $\rho$ be a quasi-periodic labelling.
Then the group $G_{\rho}$ is simple.
\end{thm}

For simplicity of notation, in what follows we will not explicitly mention the labelling $\rho$
in what we now define.
Recall that given an $x\in \mathbf{R}$ and $n\in \mathbf{N}$, we define a word $\mathcal{W}(x,n)$ as follows. 
Let $y\in \frac{1}{2}\mathbf{Z}\setminus \mathbf{Z}$ such that $x\in [y-\frac{1}{2},y+\frac{1}{2})$.
Then we define $$\mathcal{W}(x,n)=\rho(y-\frac{1}{2}n)\rho(y-\frac{1}{2}(n-1))...\rho(y)...\rho(y+\frac{1}{2}(n-1))\rho(y+\frac{1}{2} n)$$

Given a compact integer interval (i.e. with integer endpoints) $J\subset \mathbf{R}$ and $n_1,n_2\in \mathbf{N}$, we define a word $\mathcal{W}(J,n_1,n_2)$ as follows.
Let $$y_1=inf(J)+\frac{1}{2}\qquad y_2=sup(J)-\frac{1}{2}$$
Then we define $$\mathcal{W}(J,n_1,n_2)=\rho(y_1-\frac{1}{2}n_1)\rho(y_1-\frac{1}{2}(n_1-1))...\rho(y_1)...\rho(y_2)...\rho(y_2+\frac{1}{2}(n_2-1))\rho(y_2+\frac{1}{2}n_2)$$
In case $n_1=n_2=n$ we denote $\mathcal{W}(J,n_1,n_2)$ as simply $\mathcal{W}(J,n)$.

We denote by $\mathcal{W}^{-1}(x,n)$ and $\mathcal{W}^{-1}(J,n)$ as the formal inverses of the words $\mathcal{W}(x,n)$ and $\mathcal{W}(J,n)$ respectively.
We now state a few structural results about the groups $G_{\rho}$ that were proved in \cite{HydeLodha}.
For what follows, we assume that $\rho$ is a quasiperiodic labelling.

\begin{lem}\label{ElementsProperties}
Let $f\in G_{\rho}$ be a nonidentity element such that $$f=w_1...w_k\qquad w_i\in \mathbf{S}_{\rho}\text{ for }1\leq i\leq k$$
Then the following hold:
\begin{enumerate}
\item The set of breakpoints of $f$ is discrete and the set of transition points is also discrete.
\item There is an $m_f\in \mathbf{N}$ such that for any compact interval $J$ of length at least $m_f$, $f$ fixes a point in $J$.
\item For each $x\in \mathbf{R}$ and each $i\leq k$, 
$$x\cdot w_1...w_i\in [x-(k+1),x+(k+1)]$$
\end{enumerate}
\end{lem}

\begin{lem}\label{minimal0}
The action of $G_{\rho}$ on $\mathbf{R}$ is minimal.
\end{lem}

\begin{lem}\label{minimal}
For each pair of elements $m_1,m_2\in \mathbf{Z}$ and a closed interval $I\subset (m_1,m_1+1)$,  
there is a word $w_1...w_k$ in the generators $\mathbf{S}_{\rho}$ such that $$I\cdot w_1...w_k\subset (m_2,m_2+1)$$
and $$I\cdot w_1...w_i\subset [inf\{m_1,m_2\},sup\{m_1+1,m_2+1\}]$$ for each $1\leq i\leq k$.
\end{lem}

The following is an elementary corollary of the third part of Lemma \ref{ElementsProperties}.

\begin{cor}\label{images}
Let $f\in G_{\rho}$.
There is an $m\in \mathbf{N}$ such that
for any $x_1,x_2 \in \mathbf{R}$ so that $x_1-x_2\in \mathbf{Z}$, the following holds:
\begin{enumerate}
\item If $\mathcal{W}(x_1,m)=\mathcal{W}(x_2,m)$ then $$x_1-x_1\cdot f=x_2-x_2\cdot f$$
\item If $\mathcal{W}^{-1}(x_1,m)=\mathcal{W}(x_2,m)$ then $$x_1-x_1\cdot f=x_3\cdot f-x_3\qquad \text{ where } x_3=x_2\cdot \iota$$
\end{enumerate}
\end{cor}

Finally, we shall also need the following folklore result (see the Appendix in \cite{ARITHMETIC} for a proof.)

\begin{thm}\label{clF}
Every element in $F'$ can be expressed as a product of at most two commutators of elements in $F'$. 
\end{thm}

\section{A characterisation of elements of $G_{\rho}$}\label{Thecharacterisation}

The goal of this section is to establish the characterisation of elements in $G_{\rho}$ as described in the introduction (Definition \ref{Krho}).
In effect, this requires us to prove Theorem \ref{characterisation}.
Throughout this section we fix a quasi periodic labelling $\rho$.
Note that it follows from Corollary \ref{images} that $G_{\rho}\subseteq K_{\rho}$.
So much of the rest of the article shall be devoted to proving that $K_{\rho}\subseteq G_{\rho}$.
The proof of this requires us to establish some preliminary structural results about the group $G_{\rho}$.

The main structural result is Proposition \ref{uniformlystable}.
The proof of the main Theorems and Corollaries will follow from it.
Proposition \ref{uniformlystable} will be proved in a subsequent section, and its proof involves the construction of a certain family of special elements in $G_{\rho}$.

\begin{defn}
A homeomorphism $f\in \textup{Homeo}^+(\mathbf{R})$ is said to be \emph{stable} if there exists an $n\in \mathbf{N}$ 
such that the following holds. For any compact interval $I$ of length at least $n$, there is a 
nonempty open subinterval $J\subset I$ such that $J$ is pointwise fixed by $f$
and $J\cap \mathbf{Z}\neq \emptyset$.
Given a stable homeomorphism $f\in \textup{Homeo}^+(\mathbf{R})$, and an interval $[m_1,m_2]$, the restriction $f\restriction [m_1,m_2]$ is said to be an \emph{atom of $f$}, if the following holds:
\begin{enumerate}
\item $m_1,m_2\in \mathbf{Z}$.
\item There is an $\epsilon>0$ such that, for each $x\in (m_1-\epsilon,m_1+\epsilon)\cup (m_2-\epsilon, m_2+\epsilon)$,  one has
$x\cdot f=x$.
\item For any $m\in (m_1,m_2)\cap \mathbf{Z}$ and any $\epsilon>0$, there is a point $x\in (m-\epsilon, m+\epsilon)$ such that $x\cdot f\neq x$.
\end{enumerate}
Note that given a stable homeomorphism $f$, there is a unique way to express $\mathbf{R}$ as a union of 
integer intervals $\{I_{\alpha}\}_{\alpha\in P}$ such that $f\restriction I_{\alpha}$ is an atom for each $\alpha\in P$ and different 
intervals intersect in at most one endpoint. For simplicity, we will refer just to the intervals $I_{\alpha}$ as the atoms of $f$.

Given an atom $f\restriction I$, we call the intervals $[inf(I),inf(I)+1]$ and $[sup(I)-1,sup(I)]$ as the \emph{head} and the \emph{foot} of the atom, respectively.
Note that is it possible that an atom $I_{\alpha}$ has the same interval as the head and the foot, in which case $|I_{\alpha}|=1$.
Two atoms $f\restriction [m_1,m_2]$ and $f\restriction [m_3,m_4]$ are said to be \emph{conjugate} if there is an integer translation
$h(t)=t+z$ for $z\in \mathbf{Z}$ such that $$f\restriction [m_1,m_2]=h^{-1}\circ f \circ h\restriction [m_3,m_4]$$
and \emph{flip-conjugate} if there is an integer translation
$h(t)=t+z$ for $z\in \mathbf{Z}$ such that $$f\restriction [m_1,m_2]=h^{-1}\circ (\iota_{[m_1,m_2]}\circ f\circ \iota_{[m_1,m_2]})\circ h\restriction [m_3,m_4]$$
where $$\iota_{[m_1,m_2]}:[m_1,m_2]\to [m_1,m_2]$$ is the unique orientation reversing isometry.

For a fixed $n\in \mathbf{N}$ we consider the set of \emph{decorated atoms}: $$\mathcal{T}_n(f)=\{(I_{\alpha},n)\mid \alpha \in P\}$$
We say that a pair of decorated atoms
$(I_{\alpha}, n)$ and $(I_{\beta}, n)$ are equivalent if either of the following holds:
\begin{enumerate}
\item $I_{\alpha},I_{\beta}$ are conjugate  and $\mathcal{W}(I_{\alpha},n)=\mathcal{W}(I_{\beta},n)$.
\item $I_{\alpha},I_{\beta}$ are flip-conjugate  and $\mathcal{W}(I_{\alpha},n)=\mathcal{W}^{-1}(I_{\beta},n)$.
\end{enumerate}

The element $f$ is said to be \emph{uniformly stable}, if it is stable and there are finitely many equivalence classes of decorated atoms for each $n\in \mathbf{N}$.
Note that if there are finitely many equivalence classes of decorated atoms of $f$ for some $n\in \mathbf{N}$, then this holds for any $n\in \mathbf{N}$.
This is true since there are finitely many words of length $n$ in $\{a,b,a^{-1},b^{-1}\}$.

Let $\zeta$ be an equivalence class of elements in $\mathcal{T}_n(f)$.
We define the homeomorphism $f_{\zeta}$ as $$f_{\zeta}\restriction I_{\alpha}=f\restriction I_{\alpha}\text{ if }(I_{\alpha},n)\in \zeta$$
$$f_{\zeta}\restriction I_{\alpha}=id\restriction I_{\alpha}\text{ if }(I_{\alpha}, n)\notin \zeta$$
If $\zeta_1,...,\zeta_m$ are the equivalence classes of elements in $\mathcal{T}_n (f)$,
then the list of homeomorphisms $f_{\zeta_1},...,f_{\zeta_m}$ is called the \emph{cellular decomposition of $f$}.
\end{defn}

\begin{lem}\label{stabilisation}
Let $g\in K_{\rho}$.
Then there exist $g_1,g_2 \in G_{\rho}$, where $g_2$ is a commutator of elements in $G_{\rho}$, such that $g_1^{-1} (g g_2^{-1}) g_1\in K_{\rho}$ is uniformly stable.
\end{lem}

\begin{proof}
Since $g\in K_{\rho}$, we know that there is a constant $k_f$ that witnesses the conditions of Definition \ref{Krho}.
Let $x\in \mathbf{R}$ be such that $x\cdot g>x$. Since $\rho$ is quasi-periodic, there is a $y\in \mathbf{R}$ such that 
$x-y\in \mathbf{Z}$ and $$\mathcal{W}^{-1}(y,k_f)=\mathcal{W}(x,k_f)$$
It follows from $3.\textup{b}$ in Definition \ref{Krho} that $y'\cdot g<y'$ for $y'=y\cdot \iota$.
It follows that $g$ admits a fixed point $p_0\in \mathbf{R}$.
A similar conclusion is achieved starting with a point $x$ for which $x \cdot g < x$.

Assume that $p_0\in \mathbf{R}\setminus \mathbf{Z}$.
The case when $p_0\in \mathbf{R}\setminus (\frac{1}{2}\mathbf{Z}\setminus \mathbf{Z})$ is dealt with similarly.
We find an element $l_2\in F'$ such that $l_2$ is a commutator in $F'$ and $g_2=\lambda(l_2)$ coincides with $g$ on a neighborhood of $p_0$.
Note that this is possible since $p_0$ is a fixed point of $g$ and $g$ satisfies the first condition of Definition \ref{Krho}.
It follows that $gg_2^{-1}$ fixes pointwise a subinterval $I$ of nonempty interior.

Since the action of $G_{\rho}$ on $\mathbf{R}$ is minimal (see Lemma \ref{minimal}),
we can find $g_1\in G_{\rho}$ such that $0\cdot g_1^{-1}\in I$. It follows that $g_1^{-1} (g g_2^{-1}) g_1$ fixes a neighborhood of $0$. 
From an application of quasi periodicity and Definition \ref{Krho}, it follows that this element is uniformly stable.
\end{proof}

The core of the proof of Theorem \ref{characterisation} reduces to the following Proposition.

\begin{prop}\label{uniformlystable}
Consider a uniformly stable element $f\in K_{\rho}$.
There is an $n\in \mathbf{N}$ such that the following holds.
Let $\zeta_1,...,\zeta_m$ be the equivalence classes of $\mathcal{T}_n (f)$.
Then $f_{\zeta_1},...,f_{\zeta_m}\in G_{\rho}$.
In particular, since $f = f_{\zeta_1} \cdots f_{\zeta_m}$, it follows that $f\in G_{\rho}$.
\end{prop}

\section{Special elements in $G_{\rho}$}

The proof of Proposition \ref{uniformlystable} requires the construction of a certain family of \emph{special elements} in $G_{\rho}$.
We define and construct them in this section.
The construction of such elements is also a useful tool to study the groups $G_{\rho}$.
Throughout the section we assume that $\rho$ is a quasi periodic labelling.

Recall the definitions of the subgroups $\mathcal{K}, \mathcal{L} \leq G_{\rho}$ from the preliminaries.
Recall that in \cite{HydeLodha} we fixed notation for the natural isomorphisms $$\lambda:H\to \mathcal{K}\qquad \pi:H\to \mathcal{L}$$
as follows for each $f\in H, n\in \mathbf{Z}$.
 $$\lambda(f)\restriction [n,n+1]\cong_T f \qquad \text {if }\rho(n+\frac{1}{2})=b$$
$$\lambda(f)\restriction [n,n+1]\cong_T (\iota\circ f\circ \iota)\qquad \text{ if }\rho(n+\frac{1}{2})=b^{-1}$$
$$\pi(f)\restriction [n-\frac{1}{2},n+\frac{1}{2}]\cong_T f\qquad \text {if }\rho(n)=a$$
$$\pi(f)\restriction [n-\frac{1}{2},n+\frac{1}{2}]\cong_T (\iota\circ f\circ \iota)\qquad \text{ if }\rho(n)=a^{-1}$$
We also denote the naturally defined inverse isomorphisms as:
$$\lambda^{-1}:\mathcal{K}\to H\qquad \pi^{-1}:\mathcal{L}\to H$$

We consider the set of triples $$\Omega= \{a,b,a^{-1},b^{-1}\}^{<\mathbf{N}}\times \mathbf{N}\times \mathbf{N}$$
Any element $\omega\in \Omega$ is a triple represented as $(W,k_1,k_2)$.

\begin{defn}
Given an element $f\in F'$ and $\omega$ as above we define an element $\lambda_{\omega}(f)\in \textup{Homeo}^+(\mathbf{R})$ as follows.

For each $n\in \mathbf{Z}$:

$$\lambda_{\omega}(f)\restriction [n,n+1]=\lambda(f)\restriction [n,n+1]\text{ if }\begin{cases}
 \mathcal{W}([n,n+1],k_1,k_2)=W\\\text{ or } \\
 \mathcal{W}([n,n+1],k_2,k_1)=W^{-1}\end{cases}$$ 
$$ \lambda_{\omega}(f)\restriction [n,n+1]=id\restriction [n,n+1] \text{ otherwise }$$

Similarly, we define the special elements $\pi_{\omega}(f)\in \textup{Homeo}^+(\mathbf{R})$ as follows.

For each $n\in \frac{1}{2}\mathbf{Z}\setminus \mathbf{Z}$:

$$\pi_{\omega}(f)\restriction [n,n+1]=\pi(f)\restriction [n,n+1]\text{ if }\begin{cases}
 \mathcal{W}([n,n+1],k_1,k_2)=W\\ \text{ or } \\
 \mathcal{W}([n,n+1],k_2,k_1)=W^{-1}\end{cases}$$ 
$$\pi_{\omega}(f)\restriction [n,n+1]=id\restriction [n,n+1] \text{ otherwise } $$
\end{defn}

Given $\omega=(W,k_1,k_2)$ where $W=w_{-k_1}...w_0...w_{k_2}$, we call $w_0$ the \emph{central letter} of the word $W$.
\begin{remark}
Note the order of appearance of $k_1,k_2$ in $\mathcal{W}([n,n+1], \cdot , \cdot)$ in the above definition.
\end{remark}

The following is a direct consequence of the definitions.
\begin{lem}\label{SpecialElements1}
Consider $\omega_1=(W_1,k_1,k_2)$ and $\omega_2=(W_2,k_2,k_1)$
such that $W_1=W_2^{-1}$.
Then it follows that for each $f\in F'$ $$\lambda_{\omega_1}(f)=\lambda_{\omega_2}(\iota\circ f\circ \iota)\qquad \pi_{\omega_1}(f)=\pi_{\omega_2}(\iota\circ f\circ \iota)$$
In particular, by symmetry it follows that:
\begin{enumerate}
\item $\lambda_{\omega_1}(f)\in G_{\rho}$ for each $f\in F'$ if and only if $\lambda_{\omega_2}(f)\in G_{\rho}$ for each $f\in F'$.
\item $\pi_{\omega_1}(f)\in G_{\rho}$ for each $f\in F'$ if and only if $\pi_{\omega_2}(f)\in G_{\rho}$ for each $f\in F'$.
\end{enumerate}
\end{lem}

\begin{remark}
Note that $\lambda_{\omega}(f), \tau_{\omega}(f)$ for $\omega=(W,k_1,k_2)$ will be equal to the identity homeomorphism, or the trivial element of $G_{\rho}$, if $W$ does not occur as a subword of the labelling $\rho$.
If $|W|\neq k_1+k_2+1$ then these elements will also be trivial.
Finally, note that $\lambda_{\omega}(f)$ is trivial if $w_0\in \{a,a^{-1}\}$ and $\pi_{\omega}(f)$ is trivial if $w_0\in \{b,b^{-1}\}$.
\end{remark}

The key technical step in the proof of the main theorem is the following localization result.

\begin{prop}\label{SpecialElements}
Let $\omega\in \Omega$ and $f\in F'$. Then $\lambda_{\omega}(f), \pi_{\omega}(f)\in G_{\rho}$.
\end{prop}

\begin{proof}
We show this for $\lambda_{\omega}$, the proof for $\tau_{\omega}$ is similar.
Thanks to Lemma \ref{SpecialElements1}, we can assume without loss of generality that $\omega=(W,k_1,k_2)$
satisfies that the central letter of $W$ equals $b$.
As an appetizer, we first demonstrate the above proposition for $k_1,k_2\in \{0,1\}$.
The statement in its full generality will then follow using an induction on $n$ which is essentially similar to the base case.

{\bf The case $k_1=k_2=0$}. 

If $W=b$ then $\lambda_{\omega}(f)=\lambda(f)$.\\

{\bf The case $k_1=0,k_2=1$ or $k_1=1,k_2=0$}

Consider the case $W=ba$.
Given an $f\in F'$, we wish to show that $\lambda_{\omega}(f)\in G_{\rho}$.
Since $F'$ is generated by commutators, it suffices to show this in the case when $f$ is a commutator.
Since $f\in F'$, there is an $f_1\in F'$ such that $$Supp(f_1 f f_1^{-1})\subset (\frac{1}{2},1)$$
Let $f_2=f_1 f f_1^{-1}$. By self similarity of $F'$, we note that $f_2$ is a commutator in $F_{(\frac{1}{2},1)}'$. 

Let $$f_2=[f_3,f_4]\qquad \text{ for }f_3,f_4\in F_{[\frac{1}{2},1]}'\subset F_{[0,1]}'$$
Let $f_4'\in F_{[0,\frac{1}{2}]}'\subseteq F_{[0,1]}'$ such that $f_4'=h f_4 h^{-1}$ where $h(t)=t+\frac{1}{2}$. 
We claim that 
$$\lambda_{\omega}(f_2)=[\lambda(f_3),\pi(f_4')]$$
Consider an interval $[n,n+1]$ where $n\in \mathbf{Z}$.
If either $$\rho(n+\frac{1}{2})\rho(n+1)=ba$$ or $$\rho(n)\rho(n+\frac{1}{2})=a^{-1}b^{-1}$$ 
then 
$$[\lambda(f_3),\pi(f_4')]\restriction [n,n+1]=\lambda_{\omega} (f_2) \restriction [n,n+1]$$
If $\rho(n)\rho(n+\frac{1}{2})\rho(n+1)\in \{ab^{-1}a,ab^{-1}a^{-1},a^{-1}ba^{-1}\}$ then $$(Supp(\lambda(f_3))\cap [n,n+1])\bigcap (Supp (\pi(f_4'))\cap [n,n+1])=\emptyset$$
and hence 
$$[\lambda(f_3),\pi(f_4')]\restriction [n,n+1] = id\restriction [n,n+1] = \lambda_{\omega}(f_2)\restriction [n,n+1]$$
Since the central letter of $W$ is $b$, we obtain $$\lambda_{\omega}(f)=\lambda(f_1^{-1})\lambda_{\omega}(f_2)\lambda(f_1)\in G_{\rho}$$

The cases $W\in \{a^{-1}b,ab,ba^{-1}\}$ are very similar and are left as a pleasant visual exercise for the reader.

{\bf The general case}\\
We perform an induction on $sup\{k_1,k_2\}$.
Let the inductive hypothesis hold for $n\in \mathbf{N}$.
Consider a word $$W=w_{-k_1}...w_0...w_{k_2}\qquad w_i\in \{a,a^{-1},b,b^{-1}\}$$
such that $sup\{k_1,k_2\}=n+1$.
There are three cases: 
\begin{enumerate}
\item $k_2>k_1$.
\item $k_1>k_2$.
\item $k_1=k_2$.
\end{enumerate}

The first two cases are symmetric, and we deal with $k_2>k_1$ and $k_1=k_2$.\\

{\bf The case $k_2>k_1$}

Assume as above that $w_0=b$.
We wish to show that given an $f\in F'$, $\lambda_{\omega}\in G_{\rho}$.
Since $F'$ is generated by commutators, as above it suffices to show this in the case when $f$ is a commutator.

Since $f\in F'$, there is an $f_1\in F'$ such that $$Supp(f_1 f f_1^{-1})\subset (\frac{1}{2},1)$$
Let $$f_2=f_1 f f_1^{-1}$$ As before, by self similarity of $F'$, we note that $f_2$ is a commutator in $F_{(\frac{1}{2},1)}'$. 

Let $$f_2=[f_3,f_4]\qquad f_3,f_4\in F_{[\frac{1}{2},1]}'\subset F_{[0,1]}'$$
Let $f_4'\in F_{[0,\frac{1}{2}]}'\subseteq F_{[0,1]}'$ such that $f_4'=h f_4 h^{-1}$ where $h(t)=t+\frac{1}{2}$. 

We define $$W_1=w_{-k_1}...w_{-1}w_0\qquad W_2=w_1...w_{k_2}$$
$$l_1=k_1, l_2=0, l_3=0,l_4=k_2-1\qquad \omega_1=(W_1,l_1,l_2)\qquad \omega_2=(W_2,l_3,l_4)$$

Note that the central letter of $W_1$ is $b$ and the central letter of $W_2$ is $w_1\in \{a,a^{-1}\}$.
From our inductive hypothesis, we know that $\lambda_{\omega_1}(h), \lambda_{\omega_2}(h) \in G_{\rho}$ for each $h\in F'$.
One checks that:

\begin{enumerate}
\item If $w_1=a$ then $$\lambda_{\omega}(f_2)=[\lambda_{\omega_1}(f_3),\pi_{\omega_2}(f_4')]$$
\item If $w_1=a^{-1}$ then $$\lambda_{\omega}(f_2)=[\lambda_{\omega_1}(f_3),\pi_{\omega_2}(f_4'')]\qquad f_4''=\iota\circ f_4'\circ \iota$$
\end{enumerate}
Since $w_0=b$, it follows that $$\lambda_{\omega}(f)=\lambda(f_1^{-1})\lambda_{\omega}(f_2) \lambda(f_1)\in G_{\rho}$$

{\bf The case $k_1=k_2$}

Assume as above that $w_0=b$.
We wish to show that given an $f\in F'$, $\lambda_{\omega}\in G_{\rho}$.
Once again, as above it suffices to show this in the case when $f$ is a commutator.

Just as above we fix an $f_1\in F'$ such that $$Supp(f_1 f f_1^{-1})\subset (\frac{1}{2},1)$$
And let $$f_2=f_1 f f_1^{-1}$$ As before, by self similarity of $F'$, we note that $f_2$ is a commutator in $F_{(\frac{1}{2},1)}'$. 

Let $$f_2=[f_3,f_4]\qquad f_3,f_4\in F_{[\frac{1}{2},1]}'\subset F_{[0,1]}'$$
Let $f_4'\in F_{[0,\frac{1}{2}]}'\subseteq F_{[0,1]}'$ such that $f_4'=h f_4 h^{-1}$ where $h(t)=t+\frac{1}{2}$.
Let $f_5\in F'$ be an element such that $$f_5\restriction Supp(f_3)=h^{-1}(t)\restriction Supp(f_3)$$

Let $$W_1=w_{-k_1}...w_0\qquad W_2=w_1...w_{k_2}$$
and $$l_1=k_1-1,l_2=1,l_3=0,l_4=k_2-1\qquad \omega_1=(W_1,l_1,l_2)\qquad \omega_2=(W_2,l_3,l_4)$$
Note that $w_{-1},w_1$ are the central letters of $W_1,W_2$ respectively.

Let $$f_3''=\iota\circ f_3'\circ \iota \qquad \text{ if }w_{-1}=a^{-1}$$
and $$f_3''=f_3' \qquad \text{ if }w_{-1}=a$$
Let $$f_4''=\iota\circ f_4'\circ \iota\qquad \text{ if }w_{1}=a^{-1}$$ and $$f_4''=f_4'\qquad \text{ if }w_{1}=a$$

From our inductive hypothesis, we know that $\lambda_{\omega_1}(k), \lambda_{\omega_2}(k)\in G_{\rho}$ for each $k\in F'$.
One checks that $$\lambda_{\omega}(f_2)=[\lambda(f_5^{-1})\pi_{\omega_1}(f_3'')\lambda(f_5),\pi_{\omega_2}(f_4'')]$$
Since $w_0=b$, it follows that $$\lambda_{\omega}(f)=\lambda(f_1^{-1})\lambda_{\omega}(f_2) \lambda(f_1)\in G_{\rho}$$
\end{proof}

\section{The epilogue}

The goal of this section is to prove Proposition  \ref{uniformlystable},
and subsequently the 
results stated in the Introduction.
We consider a uniformly stable element $f\in K_{\rho}$.
Let $\{I_{\alpha}\}_{\alpha\in P}$ be the set of atoms of $f$.
From Definition \ref{Krho}, we know that there is a $k_f\in \mathbf{N}$ such that parts $(3.a),(3.b)$ of the Definition hold.

\begin{lem}\label{Nbhd1}
Let $f\in K_{\rho}$ and $\{I_{\alpha}\}_{\alpha\in P}$ be as above.
There is a number $l_f>k_f$ such that the following holds.
Consider $n,m \in \mathbf{Z}, \alpha\in P$ such that $[n,n+1],[m,m+1]$ are respectively the head and the foot of $I_{\alpha}$.
Assume that $n\neq m$ (and hence $I_{\alpha}$ has a distinct head and foot.)
Then it follows that $$\mathcal{W}([n,n+1], l_f)\neq \mathcal{W}([m,m+1],l_f)$$
$$\mathcal{W}([n,n+1], l_f)\neq \mathcal{W}^{-1}([m,m+1],l_f)$$
\end{lem}

\begin{proof}
First we claim that $$\mathcal{W}([n,n+1],k_f+2)\neq \mathcal{W}([m,m+1],k_f+2)$$
From the definition of the atoms of $f$, there is an $\epsilon>0$ such that $f$ fixes each point in $[n-\epsilon, n+\epsilon]$.
However, there is a point in $[m-\epsilon, m+\epsilon]$ that is moved by $f$.
It follows from Definition \ref{Krho} that either $$\mathcal{W}(n-\frac{1}{2},k_f)\neq \mathcal{W}(m-\frac{1}{2},k_f)$$
or $$\mathcal{W}(n+\frac{1}{2},k_f)\neq \mathcal{W}(m+\frac{1}{2},k_f)$$
Therefore, the claim follows.

Let $l=sup\{|I_{\beta}|\mid \beta\in P\}$.
Note that from Definition \ref{Krho} it follows that $l$ is finite.
For $l_f=k_f+l$ it follows that 
$$\mathcal{W}([n,n+1],l_f)\neq \mathcal{W}^{-1}([m,m+1],l_f)$$
(To see this, assume by way of contradiction that the equality holds.
This would imply that there is a number $t\in [n,m]$ such that $\rho(t)=\rho(t)^{-1}$, which is impossible.)
It follows that both inequalities hold for $l_f=k_f+l$.
\end{proof}

\begin{defn}\label{lf}
Given any $f\in K_{\rho}$ that is uniformly stable, we define the number emerging from the proof of the above Lemma as $$l_f=k_f+l\qquad l=sup\{|I_{\beta}|\mid \beta\in P\}$$
Note that $l_f$ satisfies both the conditions of Definition \ref{Krho}
and the conclusion of Lemma \ref{Nbhd1}.
\end{defn}

Since $f$ is uniformly stable, we can consider the cellular decomposition of $f$ as decorated atoms $\mathcal{T}_{l_f}(f)$.
Let $\zeta_1,...,\zeta_m$ be the equivalence classes of $\mathcal{T}_{l_f}(g)$.
The list of homeomorphisms $f_{\zeta_1},...,f_{\zeta_m}$ form the resulting cellular decomposition.
To prove Proposition \ref{uniformlystable} we would like to show that $f_{\zeta_1},...,f_{\zeta_m}\in G_{\rho}$.

\begin{lem}\label{Nbhd2}
Let $f\in K_{\rho}$, $\{I_{\alpha}\}_{\alpha\in P}$ and $\mathcal{T}_{l_f}(f)$ be as above.
Consider $n,m \in \mathbf{Z}, \alpha\in P$ such that:
\begin{enumerate}
\item $[n,n+1],[m,m+1]$ are subintervals of $I_{\alpha}$.
\item $[n,n+1]$ is either the head or the foot of $I_{\alpha}$ and $[m,m+1]$ is neither the head nor the foot of $I_{\alpha}$.
\end{enumerate}
Then it follows that $$\mathcal{W}([n,n+1], l_f)\neq \mathcal{W}([m,m+1],l_f)$$
$$\mathcal{W}([n,n+1], l_f)\neq \mathcal{W}^{-1}([m,m+1],l_f)$$
\end{lem}

\begin{proof}
Assume that $[n,n+1]$ is the head of $I_{\alpha}$.
(The proof for the foot is similar.)
Note that by definition, $f\restriction (n-\epsilon, n+\epsilon)=id$ for some $\epsilon>0$.
However, $f\restriction (m-\epsilon, m+\epsilon)\neq id$.
So from part $3$ of Definition \ref{Krho} our conclusion follows.
\end{proof}

\begin{defn}
An element $g\in G_{\rho}$ is said to \emph{preserve the atoms} of $f$ if the following holds:
\begin{enumerate}
\item For each $\alpha\in P$, $g$ pointwise fixes a neighborhood of $inf(I_{\alpha}), sup(I_{\alpha})$.
\item If $(I_{\alpha},l_{f})$ and $(I_{\beta},l_f)$ are equivalent, then $$g\restriction I_{\alpha}\cong_T g\restriction I_{\beta}\qquad \text{ if }\mathcal{W}(I_{\alpha},l_f)=\mathcal{W}(I_{\beta},l_f)$$ 
 $$g\restriction I_{\alpha}\cong_T \iota_{\beta}\circ g\circ \iota_{\beta} \restriction I_{\beta}\qquad \text{ if }\mathcal{W}(I_{\alpha},l_f)=\mathcal{W}^{-1}(I_{\beta},l_f)$$ 
 where $\iota_{\beta}:I_{\beta}\to I_{\beta}$ is the unique orientation reversing isometry. 
\end{enumerate}
Note that these properties are closed under composition of elements, and hence we define a subgroup of $G_{\rho}$ 
$$\mathcal{M}_f = \{ g\in G_{\rho}\mid g \text{ is atom preserving for }f\}$$
\end{defn}

Special elements in $G_{\rho}$ provide a natural source of atom preserving elements, as is observed in the proof of the Lemma below.
 
 \begin{lem}\label{atompreserving}
 The restriction $\mathcal{M}_f\restriction int(I_{\alpha})$ for each $\alpha\in P$ does not admit a global fixed point.
 \end{lem}
 
 \begin{proof}
 Let $x\in int(I_{\alpha})$.
 We would like to show the existence of an element $g\in \mathcal{M}_f$ such that $x\cdot g\neq x$.
 Let $n_1=inf(I_{\alpha}), n_2=sup(I_{\alpha})$.
 There are two cases:
 \begin{enumerate}
 \item $x\in (\frac{1}{2}\mathbf{Z}\setminus \mathbf{Z})\cap int(I_{\alpha})$.
 \item $x\in (n-\frac{1}{2}+\epsilon,n+\frac{1}{2}-\epsilon)$ for $n\in \mathbf{Z}\cap int(I_{\alpha})$ and $\epsilon>0$. 
 \end{enumerate}
 Let $g\in F'$ be an element such that $(\epsilon, 1-\epsilon)\subset Supp(g)$.
 In the latter case, from an application of Lemmas \ref{Nbhd1} and \ref{Nbhd2}, it is easy to see that the special element $\pi_{\omega_1}(g)$ for 
 $$\omega_1=(\mathcal{W}([n-\frac{1}{2},n+\frac{1}{2}],l_f), l_f, l_f)$$ 
 is atom preserving. Moreover, $x\cdot \pi_{\omega_1}(g)\neq x$, since $(\epsilon, 1-\epsilon)\subset Supp(g)$.
 
 In the former case, the special element is $\lambda_{\omega_2}(g)$ for $$\omega_2=(\mathcal{W}(x,l_f),l_f,l_f)$$ 
 is atom preserving, and $x\cdot \lambda_{\omega_2}(g)\neq x$ since $\frac{1}{2}\in (\epsilon, 1-\epsilon)\subset Supp(g)$.
 \end{proof}
 
 \begin{proof}[Proof of Proposition \ref{uniformlystable}]
 Let $f_{\zeta_{j}}$ be an element in the cellular decomposition of $f$.
 We would like to show that $f_{\zeta_j}\in G_{\rho}$.
 For each $\alpha\in P$, let $$J_{\alpha}=Supp(f_{\zeta_j})\cap I_{\alpha}$$ 
 From an application of Lemma \ref{atompreserving}, we find an element $g\in \mathcal{M}_f<G_{\rho}$ such that 
 for each $\alpha\in P$ such that $J_{\alpha}\neq \emptyset$ one of the following holds:
 \begin{enumerate}
 \item $J_{\alpha}\cdot g$ is a subset of the head of $I_{\alpha}$.
 \item  $J_{\alpha}\cdot g$ is a subset of the foot of $I_{\alpha}$.
 \end{enumerate}
 Indeed, if $\alpha,\beta\in P$ are such that $\mathcal{W}^{-1}(I_{\alpha}, l_f)=\mathcal{W}(I_{\beta},l_f)$, then 
 $J_{\alpha}\cdot g$ being a subset of the head of $I_{\alpha}$ implies that $J_{\beta}\cdot g$ is a subset of the foot of $I_{\beta}$.
 
 It follows from an application of Lemmas \ref{Nbhd1},\ref{Nbhd2} that $$g^{-1}f_{\zeta_j} g=\lambda_{\omega}(h)$$
 where $h\in F'$ and $$\omega=(\mathcal{W}(I_{\alpha}, l_f), l_f,l_f)$$
 for some (or any) $I_{\alpha}$ such that $J_{\alpha}\cdot g$ is a subset of the head of $I_{\alpha}$.
In particular, $$f_{\zeta_j}=g \lambda_{\omega}(h) g^{-1}$$
 Since by Proposition \ref{SpecialElements} $\lambda_{\omega}(h)\in G_{\rho}$, we conclude that $f_{\zeta_i}\in G_{\rho}$.
 \end{proof}
 
We can now finish the proof of Theorems \ref{characterisation}, \ref{main} and Corollaries \ref{mainprop2}, \ref{maincorollary}.

\begin{proof}[Proof of Theorem \ref{characterisation}]
We know that $G_{\rho}\leq K_{\rho}$. 
It remains to show that given $g\in K_{\rho}$, one has $g\in G_{\rho}$.
Using Lemma \ref{stabilisation}, we know that there exist $g_1,g_2\in G_{\rho}$ such that $h=g_1^{-1} (g g_2^{-1}) g_1\in K_{\rho}$ is uniformly stable.
Using Proposition \ref{uniformlystable} we conclude that $h\in G_{\rho}$.
Therefore it follows that $g\in G_{\rho}$.
\end{proof}

\begin{proof}[Proof of Proposition \ref{mainprop2}]
Let $h\in K_{\rho}=G_{\rho}$.
Thanks to Lemma \ref{stabilisation}, we know that there are elements $f_1,f_2\in G_{\rho}$ such that $f_2$ is a commutator of elements in $G_{\rho}$ and the element $f=f_1^{-1} (h f_2^{-1}) f_1$
is uniformly stable.

{\bf Claim}: There is a subgroup $K<G_{\rho}$ such that $K\cong F'\oplus...\oplus F'$ and $f\in K$.

Note that the claim implies that $$hf_2^{-1}\in f_1Kf_1^{-1} \cong F'\oplus...\oplus F' <G_{\rho}$$
So the conclusion of Proposition \ref{mainprop2} for $h$ follows from this claim.

{\bf Proof of Claim.}
We know that $f\in G_\rho$ is a uniformly stable element.
Let $\{I_{\alpha}\}_{\alpha\in P}$ be the atoms of $f$.
Let $l_f$ be the constant from Definition \ref{lf}.
Let the cellular decomposition of $f$ as decorated atoms $\mathcal{T}_{l_f}(f)$ be $f_{\zeta_1},...,f_{\zeta_m}$.
Here we represent the equivalence classes of decorated atoms in $\mathcal{T}_{l_f}(f)$ as $\zeta_1,...,\zeta_m$.
For each $1\leq i\leq m$, seet $L_i=|I_{\alpha}|$ where $(I_{\alpha},l_f)\in \zeta_i$.
(Recall that $|I_{\alpha}|=|I_{\beta}|$ whenever $(I_{\alpha},l_f),(I_{\beta},l_f)\in \zeta_i$.)
For each $1\leq i\leq m$, define the canonical isomorphism $$\phi_i:F'\to F_{[0,L_i]}'$$
where $F_{[0,L_i]}$ is the standard copy of $F$ supported on the interval $[0,L_i]$.

For each $1\leq i\leq m$, we have $$\{\mathcal{W}(I_{\alpha},l_f)\mid (I_{\alpha},l_f)\in \zeta_i\}=\{W_i,W^{-1}_i\}$$
for words $W_1,...,W_m$.
Define a map $$\phi: \oplus_{1\leq i\leq m}F'\to \textup{Homeo}^+(\mathbf{R})$$ as follows. 
For $\alpha\in P$ and $1\leq i\leq m$: $$\phi(g_1,...,g_m)\restriction I_{\alpha}\cong_T \phi_i(g_i)\qquad \text{ if }(I_{\alpha},l_f)\in \zeta_i\text{ and }\mathcal{W}(I_{\alpha},l_f)=W_i$$
$$\phi(g_1,...,g_m)\restriction I_{\alpha}\cong_T \iota_{L_i}\circ\phi_i(g_i)\circ \iota_{L_i}\qquad \text{ if }(I_{\alpha},l_f)\in \zeta_i\text{ and }\mathcal{W}(I_{\alpha},l_f)=W^{-1}_i$$
where $\iota_{L_i}:[0,L_i]\to [0,L_i]$ is the unique orientation reversing isometry. 
It is easy to check that this is an injective group homomorphism.
Moreover, the image of each element under $\phi$ satisfies the conditions of Definition \ref{Krho}.
Therefore, the image of $\phi$ lies in $K_{\rho}=G_{\rho}$ and contains $f = \phi (\phi_1^{-1}(f_{\zeta_1}), \ldots, \phi_m^{-1}(f_{\zeta_m}))$.
\end{proof}

\begin{proof}[Proof of Theorem \ref{main}]
Let $f\in G_{\rho}$.
We know from Lemma \ref{stabilisation} that there is a commutator $f_1\in G_{\rho}$ and an $f_2\in G_{\rho}$ such that $f_0=f_2(ff_1^{-1})f_2^{-1}$ is uniformly stable.
By Proposition \ref{mainprop2}, we know that there is a subgroup of $G_{\rho}$ that contains $f_0$ and is isomorphic to a direct sum of copies of $F'$.
Since by Theorem \ref{clF} every element in $F'$ can be expressed as a product of at most two commutators of elements in $F'$, the same holds for a direct sum of copies of $F'$.
It follows that $f_0$ can be expressed as a product of at most two commutators of elements in $G_{\rho}$.
Therefore, $f$ can be expressed as a product of at most three commutators of elements in $G_{\rho}$.
\end{proof}

\begin{proof}[Proof of Corollary \ref{maincorollary}]
This follows from a theorem of Ghys \cite{ghys}, according to which such an action by orientation preserving homeomorphisms of the circle  
induces a homogeneous quasimorphism (the rotation number), which is nontrivial in case of absence of a global fixed point.   
Since by Theorem \ref{main} the stable commutator length of $G_{\rho}$ vanishes, this quasimorphism must be trivial. Therefore, every such action of $G_{\rho}$ 
on $\mathbf{S}^1$ must admit a global fixed point.
\end{proof}

\begin{proof}[Proof of Corollary \ref{maincor1}]
The group $G_{\rho}$ for a quasiperiodic labelling $\rho$ cannot admit a type $\textup{(i)}$ action since it is not locally indicable (recall that $G_{\rho}$ is a simple group).
For a type $\textup{(ii)}$ action of $G_{\rho}$ it is easy to construct an element $h\in \textup{Homeo}^+(\mathbf{R})$ such that $h$ commutes with each element of $G_{\rho}$.
Upon taking a quotient, this provides a faithful fixed point free action of $G_{\rho}$ on the circle which contradicts Corollary \ref{maincorollary}.
\end{proof}


\begin{small}


\vspace{0.3cm}

\textit{James Hyde} 

Mathematical Institute,

University of St. Andrews, Scotland.

jameshydemaths@gmail.com\\

\textit{Yash Lodha} 

Institute of Mathematics, EPFL, Lausanne, Switzerland

yash.lodha@epfl.ch\\

\textit{Andr\'{e}s Navas}

Dpto. de Matem\'aticas y C.C., Universidad de Santiago de Chile

Unidad Cuernavaca Instituto de Matem\'aticas, Univ. Nac. Aut\'onoma de M\'exico

andres.navas@usach.cl

\bigskip

\textit{Crist\'obal Rivas}

Dpto. de Matem\'aticas y C.C., Universidad de Santiago de Chile

cristobal.rivas@usach.cl

\end{small}


\end{document}